\documentclass[12pt,nosumlimits,nonamelimits]{amsart}

\usepackage{geometry}
\usepackage{amsmath,amsthm,amssymb,amsfonts}
\usepackage{mathrsfs}
\usepackage{bm}
\usepackage{tikz-cd}
\usepackage{setspace}
\usepackage{enumitem}
\usepackage{hyperref}
\usepackage{titlesec}
\usepackage{cite}
\usepackage{microtype}

\titleformat{\section}{\large\bfseries}{\thesection.}{0.5em}{}
\titleformat{\subsection}{\normalsize\bfseries}{\thesubsection.}{0.5em}{}

\theoremstyle{plain}
\newtheorem{theorem}{Theorem}[section]
\newtheorem{lemma}[theorem]{Lemma}

\theoremstyle{definition}
\newtheorem{definition}[theorem]{Definition}
\newtheorem{example}[theorem]{Example}

\theoremstyle{remark}
\newtheorem{remark}[theorem]{Remark}

\newcommand{\C}{\mathbb{C}}

\DeclareMathOperator{\rk}{rank} 

\title{Non-linear traces of Choquet type on AF algebras}
\author{Ryota Ninomiya}

\address{
Department of Pure and Applied Mathematics,\\
Graduate School of Information Science and Technology,\\
The University of Osaka,\\
Yamadaoka 1--5, Suita, Osaka 565--0871, Japan
}

\email{ninomiya-r@ist.osaka-u.ac.jp}

\begin{document}
\pagenumbering{roman}

\maketitle

\begin{abstract}

We study non-linear traces of Choquet type on unital AF algebras.
Building on the matrix-algebra theory of Nagisa--Watatani and the UHF setting considered by Sogame, we first establish the corresponding description for finite direct sums of matrix algebras and then pass to arbitrary unital AF algebras.
We prove that the cone of Choquet traces on an AF algebra $A$ is isomorphic, 
as an ordered cone, to the cone of increasing maps $\alpha:\Gamma(A)\to[0,\infty)$ satisfying $\alpha(0)=0$.
On each finite-dimensional stage, the associated trace is given by an explicit Choquet formula involving spectral values and the classes of spectral projections. 
\end{abstract}

\pagenumbering{arabic}
\setstretch{1.2}

\section{Introduction}

The Choquet integral provides a flexible extension of the Lebesgue integral in which the underlying set function is allowed to be non-additive but monotone.  
In the finite-dimensional discrete setting, 
a capacity $\mu$ on a finite set $\Omega$ gives rise to a map on $[0,\infty)^\Omega$ which is characterized by three basic properties: monotonicity, positive homogeneity, and comonotonic additivity.  
This viewpoint has been extensively developed in the theory of non-additive measure and integral (see, for example, \cite{Choquet1953}, \cite{Choquet1954}, \cite{Denneberg1994}), and has found applications in decision theory, risk analysis, and probability under ambiguity.

In the operator algebraic setting, it is natural to ask whether one can develop analogous ``Choquet-type'' integration theories for positive elements of a C$^*$-algebra.  
A first step in this direction was taken by Nagisa and Watatani \cite{NagisaWatataniMatrix}, 
who introduced non-linear traces on full matrix algebras and studied them from the viewpoint of matrix analysis and majorization.  
In their framework, a non-linear trace on $M_n(\C)$ can be characterized as a Choquet map $\varphi : M_n(\C)^+ \to [0,\infty)$ which is unitarily invariant, monotone, positively homogeneous, and comonotonically additive with respect to the spectral functional calculus.
They proved that such maps admit explicit expressions in terms of the eigenvalues of a positive matrix and the ranks of its spectral projections, and that there is a one-to-one correspondence between the maps and the increasing functions on the finite dimension scale $\{0,1,\dots,n\}$.

Their results show that non-linear traces provide a natural non-commutative extension of discrete Choquet integrals to the matrix setting.  
On the one hand, every such trace is completely determined by its values on projections, and these values are encoded by an increasing function on the scale.  
On the other hand, given an increasing function on the scale, one obtains a non-linear trace via an explicit Choquet formula which involves eigenvalue gaps and the ranks of the corresponding spectral projections.
In particular, positive linear traces arise as the special case where
the scale function is linear, that is,
\[
  \alpha(k)=k\alpha(1)
  \qquad
  (k\in\{0,1,\dots,n\}).
\]
The normalized tracial state corresponds to
\[
  \alpha(k)=\frac{k}{n}.
\]

Beyond the matrix case, Nagisa and Watatani have further extended the theory of non-linear traces of Choquet type to broader operator-algebraic settings.
In particular, they studied non-linear traces on the algebras of compact operators and on semi-finite factors, where eigenvalues are replaced by generalized singular values and spectral distribution functions \cite{NagisaWatataniCompact,NagisaWatataniSemifinite}.
More recently, they investigated probabilistic aspects of such traces, including a law of large numbers for non-linear traces of Choquet type on finite factors \cite{NagisaWatataniLLN}.
We refer the reader to these works for further generalizations of non-linear traces beyond the finite-dimensional framework considered in this paper.

Sogame's master's thesis \cite{Sogame2024} studies non-linear traces on UHF C$^*$-algebras and provides an important precursor to the present work. 
In particular, the inductive-limit viewpoint in the UHF setting and the continuity estimate recalled in Lemma~\ref{lem:choquet-Lipschitz} are closely related to arguments in that thesis. 
The UHF example in Section~\ref{sec:AF-examples} should therefore be viewed as recovering that earlier setting.
The purpose of the present paper is to pass from the UHF situation to arbitrary unital AF algebras. Two additional features are needed for this passage. First, finite-dimensional stages of a general AF algebra are direct sums of full matrix algebras rather than single matrix algebras; this is the reason for the direct-sum theory developed in Section~\ref{sec:direct-sum}. 
Second, the finite-stage parameters must be made compatible with the connecting maps and with the global dimension scale $\Gamma(A)$. This leads to the AF construction and characterization in Section~\ref{sec:AF}. The commutative example $A=C(X)$ in Section~\ref{sec:AF-examples} also shows how the theory contains classical Choquet integration and supplies a non-simple AF example, a phenomenon not visible in the UHF case. 

The aim of this paper is to extend this picture from full matrix algebras to arbitrary unital AF algebras.  
For an AF algebra $A$, the ordered $K_0$-group and its distinguished scale $\Gamma(A)$ play the role of a global dimension scale which encodes the sizes of projections along an AF filtration. 

From a different viewpoint, the present work may also be regarded as a non-linear extension of the correspondence between traces on AF algebras and order-preserving states on the ordered $K_0$-group.
In the linear setting, it is well known that tracial states on an AF algebra are described in terms of the ordered $K_0$-group and its scale, see for example Davidson's book \cite{Davidson1996} (see also \cite{EffrosHandelmanShen1980}, \cite{GoodearlHandelman1976}) and the work of Blackadar on traces on simple AF algebras \cite{Blackadar1980TracesSimpleAF}.
Motivated by this perspective, we regard $\Gamma(A)$ as the natural parameter space for non-linear traces of Choquet type on a unital AF algebra.

More precisely, we consider maps $\varphi : A^+ \to [0,\infty)$ satisfying the following properties:
\begin{itemize}
  \item unitary invariance:
        $\varphi(uau^*) = \varphi(a)$ for all $a \in A^+$ and all unitaries
        $u \in A$;
  \item monotonicity:
        $0 \le a \le b$ implies $\varphi(a) \le \varphi(b)$;
  \item positive homogeneity:
        $\varphi(t a) = t\,\varphi(a)$ for all $t \ge 0$ and $a \in A^+$;
  \item comonotonic additivity on the spectrum:
        for each $a \in A^+$ and comonotonic functions
        $f,g \in C(\sigma(a))^+$, we have
        $\varphi(f(a)+g(a)) = \varphi(f(a)) + \varphi(g(a))$.
\end{itemize}
We call such maps \emph{Choquet traces} on $A$.  
When $A = M_n(\C)$, this notion reduces to the non-linear traces of Nagisa--Watatani \cite{NagisaWatataniMatrix}.  
In the AF case, the extra structure coming from the inductive limit decomposition allows one to express Choquet traces in terms of spectral data inside each finite-dimensional building block and an increasing function on the global dimension scale.

The main result of this paper identifies not only the underlying sets but also their ordered-cone structures. Let
\[
  \mathcal I(\Gamma(A))
  :=
  \{\alpha:\Gamma(A)\to[0,\infty):
       \alpha \text{ is increasing and }\alpha(0)=0\},
\]
equipped with pointwise addition, non-negative scalar multiplication, and pointwise order, and let $\mathcal C(A)$ denote the cone of Choquet traces on $A$, again with the pointwise order.
We prove that
\[
  \alpha\longmapsto \Phi_\alpha
\]
is an isomorphism of ordered cones from $\mathcal I(\Gamma(A))$ onto $\mathcal C(A)$. On every finite-dimensional subalgebra $A_n\cong\bigoplus_s M_{k_{n,s}}(\C)$, the restriction of $\Phi_\alpha$ is given by an explicit Choquet formula in terms of the spectrum and the rank vectors of spectral projections. 

The structure of the paper is as follows.
In Section~2, we review the discrete Choquet integral on a finite set and introduce the notion of a Choquet trace on a unital C$^*$-algebra.
In Section~3, we extend the theory of Nagisa--Watatani on Choquet traces from a single matrix algebra to finite direct sums of matrix algebras. 
Section~4 is devoted to the AF case: we recall the description of the dimension scale $\Gamma(A)$ of an AF algebra in terms of its inductive limit decomposition, construct Choquet traces from increasing functions on $\Gamma(A)$, and prove the characterization.
In Section~5, we illustrate the general theory by discussing concrete examples, including UHF algebras and the Fibonacci AF algebra, 
where the dimension scale and its order structure can be described explicitly.

\section{Choquet integrals and Choquet traces}

We begin by recalling the discrete Choquet integral on a finite set,
which will serve as a guiding model for the non-commutative constructions considered later.
Let $\Omega=\{1,\dots,n\}$ and 
\[
   \mu:\mathcal P(\Omega)\to[0,\infty)
\]
be an increasing set function, that is,
\[
   X\subset Y\subset\Omega \ \Longrightarrow\ \mu(X)\le\mu(Y),
   \qquad
   \mu(\emptyset)=0.
\]
Such a map is usually called a monotone measure or capacity.
For a vector $x=(x_1,\dots,x_n)\in[0,\infty)^n$, choose a permutation $\sigma$ of $\{1,\dots,n\}$ such that
\[
   x_{\sigma(1)}\ge x_{\sigma(2)}\ge\cdots\ge x_{\sigma(n)}.
\]
Set
\[
   A_i:=\{\sigma(1),\dots,\sigma(i)\},
   \qquad
   1\le i\le n.
\]
The discrete Choquet integral of $x$ with respect to $\mu$ is defined by
\[
   (C)\!\int x\,d\mu
     =
     \sum_{i=1}^{n-1}
        \bigl(x_{\sigma(i)}-x_{\sigma(i+1)}\bigr)\,\mu(A_i)
     + x_{\sigma(n)}\,\mu(A_n).
\]
This value does not depend on the choice of the permutation $\sigma$ that
orders the components in decreasing order.

Real-valued functions $f$ and $g$ on a set $\Omega$ are called \emph{comonotonic} functions if for all $s,t\in \Omega$
 \[
 (f(s)-f(t))
 (g(s)-g(t))
 \geq0.
 \]

The Choquet integral satisfies the following three properties.

\begin{itemize}
\item \textbf{Monotonicity:}
      if $0\le f\le g$, then
      $(C)\!\int f\,d\mu\le (C)\!\int g\,d\mu$.
\item \textbf{Positive homogeneity:}
      if $r\ge0$, then
      $(C)\!\int rf\,d\mu
        = r(C)\!\int f\,d\mu$.
\item \textbf{Comonotonic additivity:}
      if $f$ and $g$ are comonotonic, then
      \[
        (C)\!\int(f+g)\,d\mu
          = (C)\!\int f\,d\mu + (C)\!\int g\,d\mu.
      \]
\end{itemize}

Conversely, every map on $[0,\infty)^n$ satisfying monotonicity, positive homogeneity, and comonotonic additivity is the Choquet integral with respect to a unique capacity; see, for example, \cite{Schmeidler1986,Denneberg1994,NagisaWatataniMatrix}.
Thus the Choquet integral is characterized by monotonicity, positive homogeneity, and comonotonic additivity. 

\medskip

The characterization above suggests that the Choquet integral can be
understood abstractly as a map determined by monotonicity,
positive homogeneity, and comonotonic additivity.
Our aim is to formulate and study a non-commutative analogue of this notion
on positive elements of a C$^*$-algebra.

In the non-commutative setting, spectral data of positive elements are naturally reflected in the ordered $K_0$-theory.
Throughout the paper, for a unital C$^*$-algebra $A$ we write $K_0(A)$ for the ordered $K$-group with positive cone $K_0(A)^+$ and distinguished order unit $[1_A]\in K_0(A)^+$.

We define the \emph{dimension scale} of $A$ by
\[
  \Gamma(A):=\{\, [p]\in K_0(A)^+ : p\in A \text{ is a projection}\,\}.
\]

Motivated by Nagisa-Watatani's theory, we now introduce the notion of a Choquet trace on a C$^*$-algebra.

\begin{definition}\label{def:Choquet-trace}
Let $A$ be a unital C$^*$-algebra. 
A map $\varphi:A^+\to[0,\infty)$ is said to be a \emph{Choquet trace} if it satisfies the following properties:
\begin{itemize}
\item[(U)] \textbf{Unitary invariance}: 
for every $a\in A^+$ and every unitary $u\in A$,
\[
   \varphi(uau^*)=\varphi(a).
\]
\item[(M)] \textbf{Monotonicity}:
if $a,b\in A^+$ and $a\le b$, 
then $\varphi(a)\le\varphi(b)$.
\item[(H)] \textbf{Positive homogeneity}:
if $r\ge0$ and $a\in A^+$, then
\[
   \varphi(ra)=r\,\varphi(a).
\]
\item[(C)] \textbf{Comonotonic additivity on the spectrum}:
for each $a\in A^+$ and any continuous functions $f,g\in C(\sigma(a))^+$ that are comonotonic on $\sigma(a)$,
\[
   \varphi(f(a)+g(a))
     = \varphi(f(a))+\varphi(g(a)).
\]
Here $f(a)$ and $g(a)$ are defined by continuous functional calculus.
\end{itemize}
\end{definition}

The properties (U), (M), (H), and (C) form an exact non-commutative analogue of the defining properties of the discrete Choquet integral on a finite set. 
Property (C) expresses comonotonic additivity along the spectrum of a positive element via functional calculus.

The following lemma is proved in \cite{Sogame2024}.
We include a proof for the reader's convenience.

\begin{lemma}\label{lem:choquet-Lipschitz}
Let $A$ be a unital C$^*$-algebra and $\varphi:A^+\to[0,\infty)$ be a Choquet trace with $\varphi(0)=0$.
Then for all $a,b\in A^+$,
\[
  |\varphi(a)-\varphi(b)|
  \le \varphi(1_A)\,\|a-b\|.
\]
In particular, $\varphi$ is norm continuous on $A^+$.
\end{lemma}

\begin{proof}
Let $a,b\in A^+$ and put $\varepsilon:=\|a-b\|$.
Since $a-b$ is self-adjoint,
\[
  -\varepsilon 1_A \le a-b \le \varepsilon 1_A,
\]
and hence
\[
  a \le b+\varepsilon 1_A,
  \qquad
  b \le a+\varepsilon 1_A.
\]
By monotonicity,
\[
  \varphi(a)\le \varphi(b+\varepsilon 1_A),
  \qquad
  \varphi(b)\le \varphi(a+\varepsilon 1_A).
\]
The properties (C) and (H) imply that
\[
  \varphi(b+\varepsilon 1_A)
  = \varphi(b) + \varepsilon\varphi(1_A).
\]
Combining these inequalities yields the claim.
\end{proof}

We now recall Nagisa-Watatani's theory for the matrix algebra $A=M_n(\C)$ in \cite{NagisaWatataniMatrix}. 
Since every projection is Murray--von Neumann equivalent to a diagonal projection of a given rank, we may identify
\[
  \Gamma(A)
  \;\cong\;
  \{0,1,\dots,n\},
\]
where $i \in \{0,\dots,n\}$ corresponds to the $K_0$-class of rank-$i$ projections.
For notational convenience, we will simply write
\[
  \Gamma(A) = \{0,1,\dots,n\}
\]
in what follows.

Let 
\[
\mu_1(a)\ge \mu_2(a)\ge\cdots\ge \mu_n(a)
\]
be the eigenvalues of $a$ in decreasing order with counting multiplicity. 
Put
\[
  q_j(a) := \rk\bigl(E_a([\mu_j(a),\infty))\bigr) \in \{0,1,\dots,n\}, 
\]
where $E_a([t,\infty))$ is the spectral projection of $a$ corresponding to the interval $[t,\infty)$ for any $t\ge 0$. 
Define a map $\varphi_\alpha:A^+\to[0,\infty)$ by 
\[
  \varphi_\alpha(a)
  = \sum_{j=1}^{n}
      \bigl(\mu_j(a) - \mu_{j+1}(a)\bigr)\alpha(q_j(a)),
\]
where $\mu_{n+1}(a)=0$. 
We call $\varphi_\alpha$ the \emph{Choquet formula} associated with $\alpha$. 
This is defined and characterized in \cite{NagisaWatataniMatrix}, that is, 
the map $\varphi_\alpha$ is a Choquet trace, and conversely every Choquet trace on a matrix algebra can be represented in this form.  
There is a one-to-one correspondence between Choquet traces on $M_n(\C)$ and increasing functions on $\Gamma(A)$ that vanish at $0$.

\section{Non-linear traces of Choquet type on finite direct sums of matrix algebras}
\label{sec:direct-sum}

In this section, we extend the Nagisa-Watatani's theory from a single matrix algebra to a finite direct sum
\[
  A \;=\; \bigoplus_{s=1}^L M_{k_s}(\C).
\]
The dimension scale $\Gamma(A)$ is identified with the product
\[
  \Gamma(A)
  \;\cong\;
  \{0,1,\dots,k_1\}\times\cdots\times\{0,1,\dots,k_L\}. 
\]
We will henceforth use this identification and simply write
\[
  \Gamma(A)
  =
  \{0,1,\dots,k_1\}\times\cdots\times\{0,1,\dots,k_L\}.
\]
For a projection $p = \bigoplus_{s=1}^L p_s\in A$, we define its rank vector by
\[
  \rk(p)
  :=
  \bigl(\rk(p_1),\dots,\rk(p_L)\bigr)
  \in \Gamma(A).
\]

As in Section 2, we define the Choquet trace and the Choquet formula on direct sums of matrix algebras in this section.

\begin{definition}\label{def:Choquet-formula}
For $a\in A^+$, let
  \[
  \lambda_1(a) > \lambda_2(a) > \cdots > \lambda_{K_a}(a) \ge 0
\]
be the distinct points of the spectrum $\sigma(a)$ listed in decreasing order with $K_a=\#\sigma(a)$. 
Let $\alpha:\Gamma(A)\to[0,\infty)$ be an increasing function such that $\alpha(0)=0$. 
We define a map $\varphi_\alpha:A^+\to[0,\infty)$ by
\[
  \varphi_\alpha(a)
  = \sum_{j=1}^{K_a}
      \bigl(\lambda_j(a) - \lambda_{j+1}(a)\bigr)\alpha(r_j(a)),
\]
where $\lambda_{K_a+1}(a)=0$ and 
\[
r_i(a):=\rk\bigl(E_a([\lambda_i(a),\infty))\bigr) \in\Gamma(A). 
\]

\end{definition}

The Choquet formula in Definition~\ref{def:Choquet-formula} expresses $\varphi_\alpha(x)$ in terms of the jumps of the spectrum of $x$.
For later arguments it is convenient to allow an arbitrary finite non-increasing sequence
\[
  t_1\ge t_2\ge\cdots\ge t_m\ge0
\]
whose set of values contains $\sigma(x)$.
The following lemma shows that the same value is obtained from such a grid and the corresponding spectral projections $E_x([t_i,\infty))$.

\begin{lemma}\label{lem:proj-general}
Let $a\in A^+$ and let
\[
  t_1\ge t_2\ge\cdots\ge t_m\ge0
\]
be a finite non-increasing sequence such that
\[
  \sigma(a)\subset\{t_1,\dots,t_m\}.
\]
Put $t_{m+1}:=0$. Then
\[
  \varphi_\alpha(a)
  =
  \sum_{i=1}^m (t_i-t_{i+1})\, \alpha\!\left(\rk(E_a([t_i,\infty)))\right).
\]
\end{lemma}
\begin{proof}
Repeated values among the $t_i$ cause no difficulty: whenever $t_i=t_{i+1}$, the corresponding summand is zero. Thus, after deleting repetitions, we may assume that the sequence is strictly decreasing.
Write the distinct spectral values of $a$ as
\[
  \lambda_1>\lambda_2>\cdots>\lambda_{K_a}\ge0,
  \qquad
  \lambda_{K_a+1}:=0,
\]
and put
\[
  r_j:=\rk(E_a([\lambda_j,\infty))),
  \qquad r_0:=0.
\]
For each $i$, let
\[
  j(i):=\max\{j\in\{1,\dots,K_a\}:\lambda_j\ge t_i\},
\]
with $j(i)=0$ if the set is empty. Then
\[
  \rk(E_a([t_i,\infty)))=r_{j(i)}.
\]
Indeed, there is no spectral value of $a$ strictly between the relevant successive spectral levels.
Since every $\lambda_j$ occurs among the $t_i$, grouping the terms for which $j(i)=j$ gives a telescoping sum:
\[
\begin{aligned}
  \sum_{i=1}^m (t_i-t_{i+1}) \alpha\!\left(\rk(E_a([t_i,\infty)))\right)
  &= \sum_{j=1}^{K_a} \left( \sum_{i:\,j(i)=j}(t_i-t_{i+1}) \right)\alpha(r_j)\\
&= \sum_{j=1}^{K_a} (\lambda_j-\lambda_{j+1})\alpha(r_j)\\
&=\varphi_\alpha(a).
\end{aligned}
\]
\end{proof}

The following theorem shows that the map $\varphi_\alpha$ defined by Definition~\ref{def:Choquet-formula} is a Choquet trace on the finite-dimensional C$^*$-algebra $A=\bigoplus_{s=1}^L M_{k_s}(\C)$.

\begin{theorem}\label{thm:direct-sum-properties}
The map $\varphi_\alpha : A^+ \to [0,\infty)$ is a Choquet trace on $A$. 
\end{theorem}

\begin{proof}
We check the four properties.

\noindent
(U)  
For any positive element $a$ and a unitary element $u$ in $A$, 
we have 
\[
\sigma(a)=\sigma(uau^*), \qquad E_{uau^*}([t,\infty))=uE_{a}([t,\infty))u^*. 
\]
Hence we obtain 
\[
\lambda_j(a)=\lambda_j(uau^*), \qquad r_j(uau^*)=r_j(a). 
\]
Therefore the property (U) holds. 

\noindent
(H)  
Let $t>0$ and $a\in A^+$. 
For any $j$ and $s>0$, 
\[
\lambda_j(t a)=t\lambda_j(a), \qquad E_{t a}([s,\infty)) = E_a([s/t,\infty)). 
\]
Thus the rank vectors satisfy
\[
  r_j(t a) = r_j(a),
  \qquad
  \lambda_j(t a) - \lambda_{j+1}(t a)
    = t\bigl(\lambda_j(a) - \lambda_{j+1}(a)\bigr),
\]
and the Choquet formula gives
\[
  \varphi_\alpha(t a)
  = \sum_{j=1}^{K_a}
       \bigl(\lambda_j(t a) - \lambda_{j+1}(t a)\bigr)\,
       \alpha\bigl(r_j(t a)\bigr)
  = t \sum_{j=1}^{K_a}
       \bigl(\lambda_j(a) - \lambda_{j+1}(a)\bigr)\,
       \alpha\bigl(r_j(a)\bigr)
  = t\,\varphi_\alpha(a).
\]
The case $t=0$ is immediate from the defining formula, since $\varphi_\alpha(0)=0$.

\noindent
(C)
Let $a\in A^+$ and let $f,g\in C(\sigma(a))^+$ be comonotonic.
Write the distinct points of $\sigma(a)$ as $\lambda_1,\dots,\lambda_{K_a}$.
Since $f$ and $g$ are comonotonic on the finite set $\sigma(a)$, there is a permutation $\pi$ of $\{1,\dots,K_a\}$ such that
\[
  f(\lambda_{\pi(1)})\ge\cdots\ge
  f(\lambda_{\pi(K_a)}),
  \qquad
  g(\lambda_{\pi(1)})\ge\cdots\ge
  g(\lambda_{\pi(K_a)}).
\]
For $1\le j\le K_a$, put
\[
  e_j:=E_a(\{\lambda_{\pi(j)}\}),
  \qquad
  P_j:=e_1+\cdots+e_j.
\]
We first record a common formula for $h\in\{f,g,f+g\}$.
Set
\[
  h_j:=h(\lambda_{\pi(j)})
  \quad(1\le j\le K_a),
  \qquad
  h_{K_a+1}:=0.
\]
Then
\[
  h_1\ge h_2\ge\cdots\ge h_{K_a}\ge0.
\]
Applying Lemma~\ref{lem:proj-general} to the positive element $h(a)$ and to the non-increasing sequence $h_1,\dots,h_{K_a}$ gives
\[
  \varphi_\alpha(h(a))
  =
  \sum_{j=1}^{K_a}  (h_j-h_{j+1})\,
  \alpha\!\left(
    \rk(E_{h(a)}([h_j,\infty)))
  \right).
\]
If $h_j=h_{j+1}$, the $j$-th summand is zero.
If $h_j>h_{j+1}$, then
\[
  h_k\ge h_j \quad (k\le j),
  \qquad
  h_k\le h_{j+1}<h_j \quad (k>j).
\]
Therefore,
\[
  \{\lambda\in\sigma(a):h(\lambda)\ge h_j\}
  =
  \{\lambda_{\pi(1)},\dots,\lambda_{\pi(j)}\},
\]
and hence
\[
  E_{h(a)}([h_j,\infty))=P_j.
\]
Consequently,
\[
  \varphi_\alpha(h(a))
  =
  \sum_{j=1}^{K_a}
  (h_j-h_{j+1})\,\alpha(\rk(P_j)).
\]
Applying this formula to $h=f+g$ and using
\[
  (f+g)_j-(f+g)_{j+1}
  =
  (f_j-f_{j+1})+(g_j-g_{j+1}),
\]
we obtain
\[
\begin{aligned}
  \varphi_\alpha((f+g)(a))
  &=
  \sum_{j=1}^{K_a}
  \bigl((f_j-f_{j+1})+(g_j-g_{j+1})\bigr)
  \alpha(\rk(P_j))\\
  &=  \varphi_\alpha(f(a))+\varphi_\alpha(g(a)).
\end{aligned}
\]
Thus property \textup{(C)} holds.

\noindent
(M)
Let $0\le a\le b$ in $A^+$ such that 
\[
  a=\bigoplus_{s=1}^L a_s,
  \qquad
  b=\bigoplus_{s=1}^L b_s,
\]
with $a_s,b_s\in M_{k_s}(\C)^+$.
Then $0\le a_s\le b_s$ for each $1\le s\le L$.

Let
\[
  \mu^{(s)}_1(a_s)\ge\cdots\ge\mu^{(s)}_{k_s}(a_s),
  \qquad
  \mu^{(s)}_1(b_s)\ge\cdots\ge\mu^{(s)}_{k_s}(b_s)
\]
be the eigenvalues of $a_s$ and $b_s$, respectively, listed in decreasing order
with counting multiplicity.
By Weyl's monotonicity theorem (see \cite{Bhatia1997}), we have
\[
  \mu^{(s)}_i(a_s)\le \mu^{(s)}_i(b_s)
  \qquad(1\le i\le k_s,\ 1\le s\le L).
\]

Then, for each $1 \le s \le L$ and $t \ge 0$, 
\[
  \rk\bigl(E_{a_s}([t,\infty))\bigr)
  =
  \#\{\,i\mid \mu^{(s)}_i(a_s)\ge t\,\}
  \le
  \#\{\,i\mid \mu^{(s)}_i(b_s)\ge t\,\}
  =
  \rk\bigl(E_{b_s}([t,\infty))\bigr).
\]
Hence the rank vectors satisfy
\[
  \rk\bigl(E_a([t,\infty))\bigr)
  \le
  \rk\bigl(E_b([t,\infty))\bigr)
  \quad\text{in }\Gamma(A)
  \qquad(t\ge0).
\]

Let
\[
  T:=\sigma(a)\cup\sigma(b)\cup\{0\},
\]
and list its elements in decreasing order as
\[
  t_1>t_2>\cdots>t_m\ge0,
  \qquad t_{m+1}:=0.
\]
Applying Lemma~\ref{lem:proj-general} to $a$ and $b$ with this sequence, and using the monotonicity of $\alpha$ on $\Gamma(A)$, we obtain
\[
\begin{aligned}
  \varphi_\alpha(a)
  &= \sum_{i=1}^m (t_i-t_{i+1})\,\alpha\bigl(\rk(E_a([t_i,\infty)))\bigr) \\
  &\le \sum_{i=1}^m (t_i-t_{i+1})\,\alpha\bigl(\rk(E_b([t_i,\infty)))\bigr)
   = \varphi_\alpha(b).
\end{aligned}
\]
Thus $\varphi_\alpha$ is monotone.
\end{proof}

We now show that the above construction exhausts all Choquet traces on the finite-dimensional C$^*$-algebra $A=\bigoplus_{s=1}^L M_{k_s}(\C)$.

\begin{theorem}\label{thm:direct-sum-characterization}
Let $A=\bigoplus_{s=1}^L M_{k_s}(\C)$, and $\varphi : A^+ \to [0,\infty)$ be a Choquet trace on $A$. 
Then there exists a unique increasing function $\alpha : \Gamma(A) \to [0,\infty)$ with $\alpha(0)=0$ such that
\[
  \varphi(a) = \varphi_\alpha(a)
  \qquad (a\in A^+). 
\]
\end{theorem}

\begin{proof}
Assume that $\varphi : A^+ \to [0,\infty)$ satisfies
\textup{(U)}, \textup{(M)}, \textup{(H)} and \textup{(C)}.

For $x\in\Gamma(A)$, let $p(x)\in A$ be a projection satisfying
$\rk(p(x))=x$, and define
\[
  \alpha(x):=\varphi(p(x)).
\]

We first check that $\alpha$ is well-defined.
If $p,q\in A$ are projections with $\rk(p)=\rk(q)=x$, then $p$ and $q$ are unitarily equivalent in $A$.
By unitary invariance \textup{(U)}, we have $\varphi(p)=\varphi(q)$.
Hence $\alpha$ is well-defined.
In particular, $\alpha(0)=\varphi(0)=0$.

Next we show that $\alpha$ is increasing.
Let $x,y\in\Gamma(A)$ with $x\le y$.
Choose projections $p,q\in A$ such that $\rk(p)=x$, $\rk(q)=y$, and $p\le q$.
By monotonicity \textup{(M)},
\[
  \alpha(x)=\varphi(p)\le\varphi(q)=\alpha(y).
\]
Thus $\alpha:\Gamma(A)\to[0,\infty)$ is increasing.

Let $a\in A^+$.
To simplify notation, we denote $\lambda_j(a)$ by $\lambda_j$ for all $j$, and set $\lambda_{K_a+1}:=0$. 
For $x\in\sigma(a)$ and $1\le k\le K_a$, define
\[
  s_k(x):=(\lambda_k-\lambda_{k+1})\,\mathbf{1}_{[\lambda_k,\infty)}(x).
\]
Then each $s_k$ is monotone increasing on $\sigma(a)$. 

On $\sigma(a)$, we have the identity
\[
  x=\sum_{k=1}^{K_a} s_k(x)\qquad(x\in\sigma(a)).
\]
By continuous functional calculus,
\[
  a=\sum_{k=1}^{K_a} s_k(a).
\]

We claim that
\[
  \varphi\Bigl(\sum_{k=1}^{K_a} s_k(a)\Bigr)=\sum_{k=1}^{K_a} \varphi\bigl(s_k(a)\bigr).
\]
Indeed, set $b_1:=s_1(a)$ and $b_m:=b_{m-1}+s_m(a)$ for $m\ge2$.
Since $s_m$ and $s_1+\cdots+s_{m-1}$ are both monotone increasing on $\sigma(a)$, they are comonotonic. 
Hence by (C),
\[
  \varphi(b_m)=\varphi(b_{m-1})+\varphi\bigl(s_m(a)\bigr).
\]
By induction on $m$, we obtain the claim.

Using (H), we have
\[
  \varphi\bigl(s_k(a)\bigr)
  =(\lambda_k-\lambda_{k+1})\,\varphi\bigl(\mathbf{1}_{[\lambda_k,\infty)}(a)\bigr).
\]
Since $\mathbf{1}_{[\lambda_k,\infty)}(a)=E_a([\lambda_k,\infty))$,
\[
  \varphi\bigl(\mathbf{1}_{[\lambda_k,\infty)}(a)\bigr)
   =\alpha\bigl(\rk(E_a([\lambda_k,\infty)))\bigr)
   =\alpha\bigl(r_k(a)\bigr).
\]
Therefore,
\[
  \varphi(a)
  =\sum_{k=1}^{K_a}(\lambda_k-\lambda_{k+1})\,\alpha\bigl(r_k(a)\bigr)
  =\varphi_\alpha(a).
\]

It remains to prove uniqueness.
Suppose that $\beta:\Gamma(A)\to[0,\infty)$ is another increasing function with $\beta(0)=0$ and $\varphi=\varphi_\beta$.
For every $x\in\Gamma(A)$, choose a projection $p(x)\in A$ with $\rk(p(x))=x$. 
Since the Choquet formula gives
\[
  \varphi_\alpha(p(x))=\alpha(x),
  \qquad
  \varphi_\beta(p(x))=\beta(x),
\]
we obtain
\[
  \alpha(x)=\varphi(p(x))=\beta(x).
\]
Hence $\alpha=\beta$.
\end{proof}

\section{Non-linear traces of Choquet type on AF algebras}
\label{sec:AF}

In this section, we extend the Choquet formula from finite-dimensional
C$^*$-algebras to general AF algebras.
Throughout this section, an \emph{AF algebra} means a unital AF algebra.

Let $A$ be a unital AF algebra.
There exists an increasing sequence of finite-dimensional C$^*$-subalgebras
\[
  A_1 \subset A_2 \subset \cdots \subset A_n \subset \cdots \subset A,
  \qquad
  A = \overline{\bigcup_{n\ge1} A_n},
\]
where each $A_n$ is of the form
\[
  A_n = \bigoplus_{s=1}^{L_n} M_{k_{n,s}}(\mathbb{C}).
\]
Since $A$ is unital, we may and do assume that the inductive system $(A_n)_{n\ge1}$ consists of unital embeddings, so that the unit of each $A_n$ is identified with the unit $1_A$ of $A$.

For $m\ge n$, we denote by
\[
  \rho_{n,m}:A_n \longrightarrow A_m
\]
the canonical connecting $*$-homomorphism (the inclusion map), and we set
\[
  \rho_{n,n}:=\mathrm{id}_{A_n}.
\]
Then $\rho_{n,\ell}=\rho_{m,\ell}\circ\rho_{n,m}$ holds for $\ell\ge m\ge n$.

Each $\rho_{n,m}$ induces an order-preserving map on the dimension scales
\[
  \gamma_{n,m}:\Gamma(A_n)\longrightarrow \Gamma(A_m),
  \qquad
  \gamma_{n,m}([p]) := [\rho_{n,m}(p)]
\]
for a projection $p\in A_n$.
Then we have
\[
  \gamma_{n,n}:=\mathrm{id}_{\Gamma(A_n)}.
\]
Then $\gamma_{n,\ell}=\gamma_{m,\ell}\circ\gamma_{n,m}$ holds for
$\ell\ge m\ge n$.

It is well known that the ordered group $K_0(A)$ is the inductive limit of
$(K_0(A_n),(\rho_{n,m})_*)$ and that the dimension scale $\Gamma(A)$ is the inductive limit
\[
  \Gamma(A)=\varinjlim\bigl(\Gamma(A_n),\gamma_{n,m}\bigr)
\]
in the category of ordered sets (see \cite{Blackadar1998,Davidson1996}).
We denote by
\[
  \gamma_n:\Gamma(A_n)\longrightarrow \Gamma(A)
\]
the canonical map into the inductive limit.
By construction,
\[
  \gamma_\ell\circ \gamma_{n,\ell}=\gamma_n
  \qquad(\ell\ge n).
\]
Moreover, since the embeddings are unital, the element $\gamma_n([1_{A_n}])\in\Gamma(A)$
does not depend on $n$ and coincides with the order unit $[1_A]$.

In the finite-dimensional case, Choquet traces are completely described by increasing functions on the corresponding finite dimension scales.
This observation provides the guiding principle for the AF case.

The use of compatible finite-dimensional Choquet traces follows the same general inductive-limit idea as in the UHF setting of \cite{Sogame2024}.
In the present AF setting, however, the finite stages are arbitrary finite direct sums of matrix algebras, 
so the parameters are rank vectors and their images in the global scale $\Gamma(A)$.


Let $\alpha:\Gamma(A)\to[0,\infty)$ be an increasing map with $\alpha(0)=0$.
For each $n$, define
\[
  \alpha^{(n)}:\Gamma(A_n)\longrightarrow[0,\infty),
  \qquad
  \alpha^{(n)}(x):=\alpha(\gamma_n(x)).
\]
By section 3, there exists a unique Choquet trace
\[
  \varphi_\alpha^{(n)}:A_n^+\longrightarrow[0,\infty)
\]
such that
\[
  \varphi_\alpha^{(n)}(p)=\alpha^{(n)}(\rk (p))
\]
for any projection $p\in A_n$. 
The Choquet formula takes the form
\[
   \varphi_\alpha^{(n)}(a)
=
\sum_{k=1}^{K_a}
  (\lambda_k(a)-\lambda_{k+1}(a))\,
  \alpha^{(n)}(r_k(a)),
  \]
where $r_k(a)=\rk(E_a([\lambda_k(a),\infty)))\in\Gamma(A_n)$ and $\lambda_{K_a+1}(a):=0$.

In order to pass from the finite-dimensional algebras $A_n$ to the AF algebra $A$, we need to verify that the Choquet traces constructed at each stage are compatible with the connecting maps.
The following lemma establishes this compatibility.

\begin{lemma}\label{lem:AF-compat}
For every $n$ and every $a\in A_n^+$,
\[
  \varphi_\alpha^{(n)}(a)
  =
  \varphi_\alpha^{(n+1)}(\rho_{n,n+1}(a)).
\]
\end{lemma}

\begin{proof}
Let $a\in A_n^+$.
List the distinct spectral values of $a$ in decreasing order as
\[
  \lambda_1(a)>\cdots>\lambda_{K_a}(a)\ge0,
\]
and put $\lambda_{K_a+1}(a):=0$.
Since $\rho_{n,n+1}$ is a unital $*$-homomorphism,
we have
\[
  \sigma(\rho_{n,n+1}(a))=\sigma(a),
\]
and hence we may write
\[
  \lambda_j(\rho_{n,n+1}(a))=\lambda_j(a)
  \quad (j=1,\dots,K_a).
\]

For each $j$, the functional calculus gives
\[
  E_{\rho_{n,n+1}(a)}([\lambda_j(a),\infty))
  =
  \rho_{n,n+1}(E_a([\lambda_j(a),\infty))).
\]
Taking ranks and using the definition of $\gamma_{n,n+1}$, we obtain
\[
  \rk\!\left(E_{\rho_{n,n+1}(a)}([\lambda_j(a),\infty))\right)
  =
  \gamma_{n,n+1}\!\left(\rk(E_a([\lambda_j(a),\infty)))\right).
\]
Applying $\gamma_{n+1}$ and using $\gamma_{n+1}\circ\gamma_{n,n+1}=\gamma_n$,
we have
\[
  \gamma_{n+1}\!\left(\rk(E_{\rho_{n,n+1}(a)}([\lambda_j(a),\infty)))\right)
  =
  \gamma_n\!\left(\rk(E_a([\lambda_j(a),\infty)))\right).
\]

Therefore,
\[
\begin{aligned}
  \varphi_\alpha^{(n+1)}(\rho_{n,n+1}(a))
  &=
  \sum_{j=1}^{K_a}
    (\lambda_j(a)-\lambda_{j+1}(a))\,
    \alpha\!\left(
      \gamma_{n+1}\!\left(
        \rk(E_{\rho_{n,n+1}(a)}([\lambda_j(a),\infty)))
      \right)
    \right) \\
  &=
  \sum_{j=1}^{K_a}
    (\lambda_j(a)-\lambda_{j+1}(a))\,
    \alpha\!\left(
      \gamma_n\!\left(
        \rk(E_a([\lambda_j(a),\infty)))
      \right)
    \right) \\
  &=
  \varphi_\alpha^{(n)}(a),
\end{aligned}
\]
which proves the assertion.
\end{proof}

To extend the local constructions to the whole algebra, we need a spectral approximation lemma ensuring that spectra are preserved in the approximation process.

\begin{lemma}\label{lem:AF-spectral-approx}
Let $A$ be a unital AF algebra with an increasing sequence $(A_n)_{n\ge1}$ of finite-dimensional subalgebras whose union is dense in $A$.
For every $a\in A^+$, there exists a sequence $(a_m)_{m\ge1}\subset\bigcup_n A_n^+$ such that
\[
  \|a_m\|\le\|a\|,
  \qquad
  \|a_m-a\|\to0,
  \qquad
  \sigma(a_m)\subset\sigma(a).
\]
\end{lemma}

\begin{proof}
Since $\bigcup_n A_n$ is a dense $*$-subalgebra of $A$, its positive cone is dense in $A^+$. 
Indeed, if $c_m\in\bigcup_n A_n$ satisfies $c_m\to a^{1/2}$, then
\[
  c_m^*c_m\to a
\]
in norm.
Therefore, for each $m$, choose $n(m)$ and $b_m\in A_{n(m)}^+$ such that
\[
  \|a-b_m\|<\frac1m.
\]
Write the full spectral decomposition of $b_m$ as
\[
  b_m=\sum_{i=1}^{r_m}\lambda_i^{(m)}p_i^{(m)},
\]
where the $\lambda_i^{(m)}$ are the distinct spectral values of $b_m$, the projections $p_i^{(m)}$ are non-zero and pairwise orthogonal, and
\[
  \sum_{i=1}^{r_m}p_i^{(m)}=1_A.
\]
In particular, if $0\in\sigma(b_m)$, its spectral projection is included in this decomposition.
For each $i$, choose $\mu_i^{(m)}\in\sigma(a)$ such that
\[
  |\mu_i^{(m)}-\lambda_i^{(m)}|
  =
  \operatorname{dist}(\lambda_i^{(m)},\sigma(a)).
\]
We claim that
\[
  \operatorname{dist}(\lambda_i^{(m)},\sigma(a))
  \le \|a-b_m\|.
\]
Indeed, if
\[
  \operatorname{dist}(\lambda_i^{(m)},\sigma(a))
  >\|a-b_m\|,
\]
then $a-\lambda_i^{(m)}1_A$ is invertible and
\[
  \|(a-\lambda_i^{(m)}1_A)^{-1}\|
  =
  \frac{1}{
    \operatorname{dist}(\lambda_i^{(m)},\sigma(a))
  }.
\]
Hence
\[
  \|
  (a-\lambda_i^{(m)}1_A)^{-1}(b_m-a)
  \|<1.
\]
By the Neumann series,
\[
  1+  (a-\lambda_i^{(m)}1_A)^{-1}(b_m-a)
\]
is invertible. 
Since
\[
  b_m-\lambda_i^{(m)}1_A
  =
  (a-\lambda_i^{(m)}1_A)
  \left(
    1+
    (a-\lambda_i^{(m)}1_A)^{-1}(b_m-a)
  \right),
\]
this would imply that $b_m-\lambda_i^{(m)}1_A$ is invertible, contradicting $\lambda_i^{(m)}\in\sigma(b_m)$.
Therefore
\[
  \operatorname{dist}(\lambda_i^{(m)},\sigma(a))
  \le \|a-b_m\|.
\]
Define
\[
  a_m:=\sum_{i=1}^{r_m}\mu_i^{(m)}p_i^{(m)}
  \in A_{n(m)}^+.
\]
Because the projections $p_i^{(m)}$ sum to $1_A$, no additional zero spectral value is introduced. 
Hence
\[
  \sigma(a_m)
  \subset
  \{\mu_1^{(m)},\dots,\mu_{r_m}^{(m)}\}
  \subset\sigma(a).
\]
Moreover, since the $p_i^{(m)}$ are pairwise orthogonal,
\[
\begin{aligned}
  \|a_m-b_m\|
  &=
  \left\|
    \sum_{i=1}^{r_m}
    (\mu_i^{(m)}-\lambda_i^{(m)})p_i^{(m)}
  \right\|\\
  &=
  \max_{1\le i\le r_m}
  |\mu_i^{(m)}-\lambda_i^{(m)}|\\
  &\le \|a-b_m\|\\
  &<\frac1m.
\end{aligned}
\]
Therefore
\[
  \|a_m-a\|
  \le \|a_m-b_m\|+\|b_m-a\|
  <\frac2m.
\]
Finally, every $\mu_i^{(m)}$ belongs to $\sigma(a)\subset[0,\|a\|]$, so
\[
  \|a_m\|=\max_i\mu_i^{(m)}\le\|a\|.
\]
\end{proof}

To extend the locally defined maps to the whole algebra $A$, set
\[
  L:=\alpha([1_A]).
\]
Since
\[
  \varphi_\alpha^{(n)}(1_A)
  =
  \alpha^{(n)}([1_{A_n}])
  =
  \alpha([1_A])
  =
  L,
\]
Lemma~\ref{lem:choquet-Lipschitz} shows that every $\varphi_\alpha^{(n)}$ is $L$-Lipschitz.
If $a\in A_n^+$ and $b\in A_m^+$, choose $k\ge n,m$. By Lemma~\ref{lem:AF-compat}, both values may be computed in $A_k$, and hence
\[
  |\varphi_\alpha^{(n)}(a)-\varphi_\alpha^{(m)}(b)|
  \le L\|a-b\|.
\]
Thus the induced map on $\bigcup_n A_n^+$ is well defined and $L$-Lipschitz.

We now construct a global Choquet trace on $A$ from an increasing function on
the dimension scale.

\begin{theorem}\label{thm:AF-existence}
Let $A$ be an AF algebra and
$\alpha:\Gamma(A)\to[0,\infty)$ be an increasing map with $\alpha(0)=0$.
Then there exists a unique Choquet trace
\[
  \Phi_\alpha:A^+\longrightarrow[0,\infty)
\]
such that
\[
  \Phi_\alpha(p)=\alpha([p]) \qquad \text{for every projection } p\in A.
\]
\end{theorem}

\begin{proof}
Define a map
\[
  f:\bigcup_{n\ge1}A_n^+ \to[0,\infty),
  \qquad
  f(a):=\varphi_\alpha^{(n)}(a)\quad (a\in A_n^+).
\]
By Lemma~\ref{lem:AF-compat}, $f$ is well defined.
Moreover, each $\varphi_\alpha^{(n)}$ is $L$-Lipschitz with
\[
  L=\alpha([1_A]),
\]
hence $f$ is $L$-Lipschitz on $\bigcup_n A_n^+$.
Since $\bigcup_n A_n^+$ is norm-dense in $A^+$, $f$ extends uniquely to a
$L$-Lipschitz map
\[
  \Phi_\alpha:A^+\to[0,\infty).
\]
We next verify the prescribed values on arbitrary projections.
Let $p\in A$ be a projection.
Choose self-adjoint elements $x_k\in\bigcup_n A_n$ with $\|x_k-p\|\to0$.
For all sufficiently large $k$, the spectrum of $x_k$ is contained in two disjoint neighborhoods of $0$ and $1$, so
\[
  q_k:=\mathbf 1_{(1/2,\infty)}(x_k)
\]
is defined by continuous functional calculus on $\sigma(x_k)$ and is a projection belonging to the same finite-dimensional stage as $x_k$.
Moreover, $q_k\to p$ in norm.
For large $k$, $\|q_k-p\|<1$, and therefore $q_k$ and $p$ are unitarily equivalent. 
Hence $[q_k]=[p]$ in $K_0(A)$.
Since $q_k$ belongs to a finite stage,
\[
  \Phi_\alpha(q_k)
  =
  \alpha([q_k])
  =
  \alpha([p]).
\]
By the norm continuity of $\Phi_\alpha$,
\[
  \Phi_\alpha(p)
  =
  \lim_k\Phi_\alpha(q_k)
  =
  \alpha([p]).
\]

\noindent
(U) 
Let $u\in A$ be unitary and $a\in A^+$.
Choose a sequence $n(k)$ and elements $u_k\in A_{n(k)}$ and $a_k\in A_{n(k)}^+$
such that $\|u-u_k\|\to0$ and $\|a-a_k\|\to0$.
For sufficiently large $k$, replace $u_k$ by its polar part $u_k(u_k^*u_k)^{-1/2}$, so that
$u_k$ is unitary and still $u_k\to u$ in norm.
Since $\varphi_\alpha^{(n(k))}$ is unitarily invariant on $A_{n(k)}^+$,
\[
  \Phi_\alpha(a_k)
  =\varphi_\alpha^{(n(k))}(a_k)
  =\varphi_\alpha^{(n(k))}(u_k a_k u_k^*)
  =\Phi_\alpha(u_k a_k u_k^*).
\]
Letting $k\to\infty$ and using norm-continuity of $\Phi_\alpha$, we obtain
$\Phi_\alpha(a)=\Phi_\alpha(uau^*)$.

\noindent
(M)  
Let $0\le a\le b$ in $A$.
Choose $n(k)$ and $a_k,b_k\in A_{n(k)}^+$ such that $\|a-a_k\|\to0$ and
$\|b-b_k\|\to0$.
Fix $\varepsilon>0$. For sufficiently large $k$, we have $\|a-a_k\|<\varepsilon$ and
$\|b-b_k\|<\varepsilon$. 
Since $a\le b$, it follows that
\[
  a_k \le b_k + 2\varepsilon 1_A
\]
for sufficiently large $k$.
Using monotonicity and positive homogeneity on $A_{n(k)}^+$,
\[
  \Phi_\alpha(a_k)=\varphi_\alpha^{(n(k))}(a_k)
  \le \varphi_\alpha^{(n(k))}(b_k+2\varepsilon 1_A)
  = \Phi_\alpha(b_k+2\varepsilon 1_A).
\]
Letting $k\to\infty$ and using continuity, we obtain
\[
  \Phi_\alpha(a)\le \Phi_\alpha(b+2\varepsilon 1_A).
\]
Since $\varepsilon$ is arbitrary, we have $\Phi_\alpha(a)\le\Phi_\alpha(b)$.

\noindent
(H)  
Let $r\ge0$ and $a\in A^+$.
Choose $a_k\in\bigcup_n A_n^+$ with $\|a-a_k\|\to0$.
Then for each $k$,
\[
  \Phi_\alpha(ra_k)=\varphi_\alpha^{(n(k))}(ra_k)=r\,\varphi_\alpha^{(n(k))}(a_k)=r\,\Phi_\alpha(a_k).
\]
Letting $k\to\infty$, we obtain $\Phi_\alpha(ra)=r\,\Phi_\alpha(a)$.

\noindent
(C)  
Let $a\in A^+$ and any continuous functions $f,g\in C(\sigma(a))^+$ that are comonotonic on $\sigma(a)$. 
Choose a sequence $a_m\in\bigcup_n A_n^+$ such that
\[
  \|a_m-a\|\to0,
  \qquad
  \sigma(a_m)\subset\sigma(a)
\]
(Lemma~\ref{lem:AF-spectral-approx}).
Then $f$ and $g$ are also comonotone on $\sigma(a_m)$.

For $m$, let $n(m)$ with $a_m\in A_{n(m)}^+$.
Since the property (C) of $\varphi_\alpha^{(n(m))}$, 
we have
\begin{align*}
\Phi_\alpha\bigl(f(a_m)+g(a_m)\bigr)
  &=\varphi_\alpha^{(n(m))}\bigl(f(a_m)+g(a_m)\bigr)
\\
&=\varphi_\alpha^{(n(m))}\bigl(f(a_m)\bigr)+\varphi_\alpha^{(n(m))}\bigl(g(a_m)\bigr)
\\
&=\Phi_\alpha\bigl(f(a_m)\bigr)+\Phi_\alpha\bigl(g(a_m)\bigr).
\end{align*}
Choose continuous extensions $\widetilde f,\widetilde g$ of
$f$ and $g$ to $[0,\|a\|]$. Since
$\sigma(a_m)\subset\sigma(a)$, we have
\[
  f(a_m)=\widetilde f(a_m),
  \qquad
  g(a_m)=\widetilde g(a_m).
\]
By continuity of the continuous functional calculus,
\[
  f(a_m)\to f(a),
  \qquad
  g(a_m)\to g(a)
\]
in norm.
By continuity of $\Phi_\alpha$, letting $m\to\infty$ yields
\[ \Phi_\alpha\bigl(f(a)+g(a)\bigr)=\Phi_\alpha\bigl(f(a)\bigr)+\Phi_\alpha\bigl(g(a)\bigr).
\]
Therefore $\Phi_{\alpha}$ is a Choquet trace. 

Finally, we show that the map $\Phi_\alpha$ is unique.
Let $\Psi:A^+\to[0,\infty)$ be a Choquet trace with
$\Psi(p)=\alpha([p])$ for every projection $p\in A$.
By Lemma~$\ref{lem:choquet-Lipschitz}$, $\Psi$ is Lipschitz continuous. 
Fix $n$, the characterization of Choquet traces on $A_n$ shows that $\Psi|_{A_n^+}=\varphi_\alpha^{(n)}$.
Hence $\Psi=f$ on $\bigcup_n A_n^+$, and by continuity, 
we conclude $\Psi=\Phi_\alpha$ on $A^+$.
\end{proof}

We next show that every Choquet trace on $A$ arises uniquely from an increasing function on the dimension scale.

\begin{lemma}\label{lem:AF-from-Phi-to-alpha}
Let $\varphi:A^+\to[0,\infty)$ be a Choquet trace on a unital AF algebra $A$. 
Then there exists a unique increasing map
\[
  \alpha:\Gamma(A)\to[0,\infty),
  \qquad
  \alpha(0)=0,
\]
such that
\[
  \alpha([p])=\varphi(p)
\]
for every projection $p\in A$.
\end{lemma}

\begin{proof}
For AF algebras, comparison of projections is detected by the ordered $K_0$-group: projections with the same $K_0$-class are unitarily equivalent, and $[p]\le[q]$ implies that $p$ is unitarily equivalent to a subprojection of $q$; see, for example \cite{Blackadar1998,Davidson1996}.

Define
\[
  \alpha([p]):=\varphi(p).
\]
If $[p]=[q]$, then $p$ and $q$ are unitarily equivalent, so property \textup{(U)} gives
\[
  \varphi(p)=\varphi(q).
\]
Hence $\alpha$ is well defined.
Also,
\[
  \alpha(0)=\varphi(0)=0.
\]
If $[p]\le[q]$, choose a projection $p'\le q$ which is unitarily equivalent to $p$. By $\textup{(U)}$ and $\textup{(M)}$,
\[
  \alpha([p])= \varphi(p) = \varphi(p')
  \le \varphi(q)=\alpha([q]).
\]
Thus $\alpha$ is increasing.
Uniqueness follows immediately from its values on projection classes.
\end{proof}

\begin{theorem}\label{thm:AF-characterization}
Let
\[
  \mathcal I(\Gamma(A))
  :=
  \{\alpha:\Gamma(A)\to[0,\infty):
    \alpha \text{ is increasing and }\alpha(0)=0\},
\]
and let
\[
  \mathcal C(A)
  :=
  \{\text{Choquet traces on }A\}.
\]
Equip both sets with pointwise addition, non-negative scalar multiplication, and pointwise order.
Then
\[
  F:\mathcal I(\Gamma(A))\longrightarrow\mathcal C(A),
  \qquad
  F(\alpha):=\Phi_\alpha,
\]
is an isomorphism of ordered cones.
In particular,
\[
  \Phi_{\alpha+\beta}
  =
  \Phi_\alpha+\Phi_\beta,
  \qquad
  \Phi_{c\alpha}
  =
  c\,\Phi_\alpha
  \quad(c\ge0),
\]
and
\[
  \alpha\le\beta
  \quad\Longleftrightarrow\quad
  \Phi_\alpha\le\Phi_\beta.
\]
\end{theorem}

\begin{proof}
By Theorem~\ref{thm:AF-existence} and Lemma~\ref{lem:AF-from-Phi-to-alpha}, the map $F$ is bijective. Indeed, every Choquet trace $\varphi$ determines a unique increasing map $\alpha$ by \[   \alpha([p])=\varphi(p) \] for projections $p\in A$, and the uniqueness part of Theorem~\ref{thm:AF-existence} gives $\varphi=\Phi_\alpha$.
If $\alpha,\beta\in\mathcal I(\Gamma(A))$, then
$\Phi_\alpha+\Phi_\beta$ is again a Choquet trace and, for every projection $p\in A$,
\[
  (\Phi_\alpha+\Phi_\beta)(p)
  =
  \alpha([p])+\beta([p])
  =
  (\alpha+\beta)([p]).
\]
The uniqueness part of Theorem~\ref{thm:AF-existence} therefore gives
\[
  \Phi_{\alpha+\beta}
  =
  \Phi_\alpha+\Phi_\beta.
\]
The same argument gives
\[
  \Phi_{c\alpha}=c\Phi_\alpha
\]
for every $c\ge0$.
Suppose next that $\alpha\le\beta$.
For $a\in A_n^+$, the finite-stage Choquet formulas and the non-negativity of the coefficients $\lambda_j(a)-\lambda_{j+1}(a)$ give
\[
  \varphi_\alpha^{(n)}(a)
  \le
  \varphi_\beta^{(n)}(a).
\]
By density of $\bigcup_n A_n^+$ and continuity,
\[
  \Phi_\alpha(a)\le\Phi_\beta(a)
  \qquad(a\in A^+).
\]
Thus $F$ is order preserving.
Conversely, if $\Phi_\alpha\le\Phi_\beta$, then for every projection $p\in A$,
\[
  \alpha([p])
  =
  \Phi_\alpha(p)
  \le
  \Phi_\beta(p)
  =
  \beta([p]).
\]
Since every element of $\Gamma(A)$ is represented by a projection, $\alpha\le\beta$.
Hence $F$ and its inverse are order preserving.
\end{proof}

\begin{remark}
Consequently, if the dimension scales of two unital AF algebras are order-isomorphic as ordered sets with distinguished endpoints, then their cones of Choquet traces are order-isomorphic as ordered cones.
Thus the ordered cone of Choquet traces cannot distinguish two AF algebras whose dimension scales are merely different realizations of the same ordered set.
\end{remark}

\begin{remark}\label{rem:additivity-subadditivity}
The characterization above also gives a simple criterion for
additivity. Namely, $\Phi_\alpha$ is additive on $A^+$ if and only if
$\alpha$ is additive on the dimension scale in the sense that
\[
  \alpha(x+y)=\alpha(x)+\alpha(y)
\]
whenever $x,y,x+y\in\Gamma(A)$.
Since $A$ is an AF algebra, the refinement and comparison properties of projections allow us to choose orthogonal projections $p,q\in A$ such that
\[
  [p]=x,\qquad [q]=y,\qquad [p+q]=x+y.
\]
Hence the only if implication follows immediately from the
additivity of $\Phi_\alpha$.
For the converse implication, at each finite-dimensional stage
\[
  A_n=\bigoplus_{s=1}^{L_n}M_{k_{n,s}}(\C),
\]
such an $\alpha^{(n)}$ is necessarily of the form
\[
  \alpha^{(n)}(x_1,\ldots,x_{L_n})
  =
  \sum_{s=1}^{L_n}c_sx_s
\]
for some $c_s\ge0$, and hence the corresponding Choquet trace is
\[
  \varphi_\alpha^{(n)}
  \left(\bigoplus_{s=1}^{L_n}a_s\right)
  =
  \sum_{s=1}^{L_n}c_s\operatorname{Tr}(a_s).
\]
Thus, in this case, $\Phi_\alpha$ is the restriction to $A^+$ of a
positive linear tracial functional on $A$.

For subadditivity, a necessary condition is
\[
  \alpha(x+y)\le\alpha(x)+\alpha(y)
\]
whenever $x,y,x+y\in\Gamma(A)$.
A complete characterization of subadditivity in terms of the dimension
scale for an arbitrary AF algebra is less immediate, and we leave this
question open.
\end{remark}

\section{Examples}
\label{sec:AF-examples}

We illustrate Theorem~\ref{thm:AF-characterization} by several concrete examples of dimension scales arising from AF algebras.
In each case, Choquet traces are described in terms of increasing maps on the corresponding scale.
For the computations of the $K_0$-groups and the associated dimension scales, 
we refer to \cite{Davidson1996}, \cite{Blackadar1998}.

\begin{example}[UHF algebras]
Let $A$ be a UHF algebra of type
\[
   N = \prod_p p^{n_p},
   \qquad n_p \in \mathbb{N}\cup\{\infty\}.
\]
It is well known that $A$ can be realized as an inductive limit of full
matrix algebras
\[
   A_k\cong M_{d_k}(\mathbb{C}),
   \qquad
   d_k\mid d_{k+1},
\]
where the supernatural number associated with the sequence $(d_k)$ is $N$. 
The connecting maps may be chosen to be unital.
At each finite stage $k$, the dimension scale is given by
\[
   \Gamma(A_k)=\{0,1,\dots,d_k\}.
\]

Let $\mathbb{Q}_N$ denote the subgroup of\/ $\mathbb{Q}$ consisting of
rational numbers whose denominators divide $N$.
In the inductive limit, the dimension scale $\Gamma(A)$ can be identified,
as an ordered subset of $\mathbb{R}$, with
\[
   \Gamma(A)\cong [0,1]\cap\mathbb{Q}_N,
\]
via the correspondence
\[
   \gamma_k(\rk(p)) \longmapsto \frac{\rk(p)}{d_k},
   \qquad p\in A_k,
\]
where the order unit corresponds to $1$. 

The algebra $A$ admits a unique tracial state $\tau$.
By Theorem~\ref{thm:AF-characterization}, Choquet traces on $A$ are in
bijection with increasing maps on $\Gamma(A)$.
In particular, any increasing function
\[
   f:[0,1]\to[0,\infty),
   \qquad f(0)=0,
\]
induces an increasing map on the scale via
\[
   \alpha_f([p]) := f(\tau(p)),
\]
and hence a Choquet trace $\Phi_{\alpha_f}$ on $A$.
In general, unless $f$ is linear, we have
\[
   \Phi_{\alpha_f}(a) \neq f(\tau(a))
\]
for $a\in A^+$.
\end{example}

\begin{example}[The Fibonacci AF algebra]
Let $A$ be the Fibonacci AF algebra, whose Bratteli diagram has incidence
matrix
\[
   F=
   \begin{pmatrix}
      1 & 1\\
      1 & 0
   \end{pmatrix}
\]
\cite{EffrosShen1980}. 
The ordered $K_0$-group of $A$ is given by the inductive limit
\[
   K_0(A)\cong \varinjlim(\mathbb{Z}^2,F^{\mathsf T}),
\]
and the dimension scale $\Gamma(A)$ consists of the classes of projections
in $K_0(A)^+$.

The algebra $A$ is simple and admits a unique tracial state $\tau$. 
Let
\[
  \phi:=\frac{1+\sqrt5}{2}.
\]
With the standard normalization of the order unit, the dimension scale can be realized as
\[
  \Gamma(A)\cong [0,1]\cap\mathbb Z[\phi].
\]

As in the UHF case, Choquet traces on $A$ correspond to increasing maps on $\Gamma(A)$. However, in the infinite-dimensional UHF case, the ordered sets
\[
  [0,1]\cap\mathbb Q_N
  \qquad\text{and}\qquad
  [0,1]\cap\mathbb Z[\phi]
\]
are order-isomorphic. Consequently, by Theorem~\ref{thm:AF-characterization}, the corresponding ordered cones of Choquet traces are order-isomorphic.
Thus these two examples do not distinguish the underlying AF algebras by their ordered cones of Choquet traces.
\end{example}

\begin{example}[A construction from a tracial state]
Let $A$ be a unital AF algebra admitting a tracial state $\tau$.
For $\lambda>0$, define an increasing map on $\Gamma(A)$ by
\[
  \alpha_\lambda([p]) := \tau(p)^\lambda
\]
for any projection $p$.
Since $t\mapsto t^\lambda$ is increasing on $[0,1]$ for every $\lambda>0$, the map $\alpha_\lambda$ is increasing.
By Theorem~\ref{thm:AF-characterization}, this yields a Choquet trace $\Phi_{\alpha_\lambda}$ on $A$.
When $\lambda\neq1$, the resulting Choquet trace is non-linear and is not determined solely by the value of $\tau(a)$ for general $a\in A^+$.
\end{example}

\begin{example}[The commutative case]
Let $X$ be a compact metrizable totally disconnected space and let
\[
  A=C(X).
\]
Then $A$ is a unital AF algebra. 
Every projection in $C(X)$ is of the form $\mathbf 1_U$, where $U\subset X$ is both open and closed.
Hence the dimension scale $\Gamma(C(X))$ may be identified with the ordered set of subsets of $X$ that are both open and closed, ordered
by inclusion.

Let $\alpha:\Gamma(C(X))\to[0,\infty)$ be increasing with $\alpha(0)=0$. 
For every subset $U\subset X$ that is both open and closed, define
\[
  \nu(U):=\alpha([\mathbf 1_U]).
\]

If $h\in C(X)^+$ is locally constant and has distinct values
\[
  t_1>\cdots>t_m\ge0,
  \qquad t_{m+1}=0,
\]
then
\[
  U_j:=\{x\in X:h(x)\ge t_j\}
\]
is both open and closed, and
\[
  \Phi_\alpha(h)
  =
  \sum_{j=1}^m
  (t_j-t_{j+1})\nu(U_j).
\]
Thus the finite-stage formula agrees with the classical Choquet formula on locally constant positive functions.

Since locally constant functions are uniformly dense in $C(X)$,
$\Phi_\alpha$ is the unique continuous extension of these Choquet-type integrals to $C(X)^+$.

In particular, if $\nu$ is the restriction to the subsets that are both open and closed of a capacity $\mu$ defined on the Borel subsets of $X$, 
then the above formula agrees with the classical Choquet integral with respect to $\mu$ on locally constant positive functions.
Since both constructions are continuous with respect to the uniform norm, 
they agree on all of $C(X)^+$ whenever the classical Choquet integral associated with $\mu$ is considered on continuous positive functions.

If $X$ has more than one point, then $C(X)$ is not simple. 
Thus this example also provides a non-simple AF algebra.
\end{example}

These examples illustrate how the order-theoretic structure of the
dimension scale $\Gamma(A)$ governs the possible classes of Choquet traces
on AF algebras.

\section*{Acknowledgements}
The author would like to express his sincere gratitude to Associate Professor Norio Nawata for his valuable advice, helpful discussions, and continuous encouragement throughout the preparation of this paper.

The author used ChatGPT, developed by OpenAI, for assistance with English-language editing. 
The author takes full responsibility for the content of the manuscript.

\end{document}